\documentclass[a4paper,12pt]{article}

\usepackage[left=2.4cm,right=2.4cm,top=2.5cm,bottom=2.8cm, marginparsep=-1mm]{geometry}

\usepackage{times,mathrsfs}
\usepackage{cmbright}
\usepackage{microtype}
\usepackage{lmodern}
\usepackage[T1]{fontenc}
\usepackage[utf8]{inputenc}
\usepackage[english]{babel}

\usepackage{graphicx} 
\usepackage{verbatim} 
\usepackage{amsmath,amssymb,amsthm}
\usepackage{color}
\usepackage{enumerate}
\usepackage{array}
\usepackage[noadjust]{cite}
\usepackage{multirow}

\usepackage{caption}
\usepackage{subfig}
\usepackage{wrapfig}

\usepackage{tikz}
\usetikzlibrary{calc}
\usetikzlibrary{arrows}
\usetikzlibrary{decorations.pathreplacing}  

\usepackage[pdftex]{hyperref}
\definecolor{darkgreen}{rgb}{0,0.4,0}
\definecolor{BrickRed}{rgb}{0.65,0.08,0}
\hypersetup{colorlinks=true,linkcolor=blue,citecolor=red,filecolor=BrickRed,urlcolor=darkgreen}

\newcommand{\LandauO}{\mathcal{O}}

\newcommand{\PR}{\mathbb{P}} 
\newcommand{\E}{\mathbb{E}} 
\newcommand{\V}{\mathbb{V}}

\newcommand{\Rc}{\mathcal{R}}
\newcommand{\Sc}{\mathcal{S}}
\newcommand{\Tc}{\mathcal{T}}

\newcommand{\N}{\mathbb{N}}

\newtheorem{theorem}{Theorem}[section]
\newtheorem{lemma}[theorem]{Lemma}
\newtheorem{proposition}[theorem]{Proposition}
\newtheorem{corollary}[theorem]{Corollary}
\newtheorem{definition}[theorem]{Definition}
\theoremstyle{remark}
\newtheorem{remark}[theorem]{Remark}

\newtheoremstyle{conjecture}{}{}{\it}{}{\color{purple}\bfseries}{}{ }{}
\theoremstyle{conjecture}

\newcommand{\OEIS}[1]{\href{http://oeis.org/#1}{OEIS~#1}}
\newcommand{\OEISs}[1]{\href{http://oeis.org/#1}{#1}}

\newcommand{\DD}{D}
\newcommand{\CC}{\mathcal{C}}

\newcommand{\Ai}{\text{\normalfont Ai}}

\begin{document}

\title{\textbf{The decompressed tree size of $k$-ary chains}}
\date{}

\author{Michael Wallner\thanks{Institute of Discrete Mathematics, TU Graz, Austria \emph{and} Institute of Discrete Mathematics and Geometry, TU Wien, Austria }} 

\maketitle

\begin{abstract}
	A chain is defined as a directed acyclic graph (DAG) with one source and one sink, where the children are ordered and the spanning tree computed using a depth-first search is a path. 
Such DAGs emerge in the context of tree compression and are therefore uniquely associated with a tree. 
The tree size of a DAG is defined as the size of the associated tree. 
For fixed out-degree $k \geq 2$, we compute the asymptotic expected decompressed tree size of a chain of size $n$ chosen uniformly at random, and we show that it contains a stretched exponential term of the form $e^{c \, \sqrt{n}}$. 
This result also has implications for the limit distribution of Brauer chains of fixed length.

\medskip

\end{abstract}

\section{Introduction}
\label{sec:intro}

Trees are a fundamental data structure in computer science.
To save memory, many effective applications store them as directed acyclic graphs (DAGs) such that repeated occurrences of the same subtrees are only stored once; see Figure~\ref{fig:compression}. 
This process is known as \emph{sharing} or \emph{DAG compression} and is used in 
XML documents~\cite{bousquet2015xml},
binary decision diagrams~\cite{MeinelTheobald1998Algorithms},
compilers~\cite{DowneySethiTarjan1980variations}, 
and various programming languages~\cite{Aho1986compilers}.
This gives a bijection between a tree class and a DAG class.
The compression is very efficient and can be implemented in worst-case time $O(n)$ 
~\cite{DowneySethiTarjan1980variations}.
The gain in memory has been analyzed previously~\cite{bousquet2015xml,flss90}: 
A uniformly chosen simply generated tree (e.g., a binary tree) of size~$n$ has on average a compressed size (i.e., the size of the DAG) that is asymptotically equivalent to
\[
	\frac{C' \, n}{\sqrt{\log n}},
\]
with a model-dependent constant $C'$.

So far, nothing was known about the reverse question: What is the average \emph{decompressed size} of a compressed tree chosen uniformly at random among compressed trees of the same size?
In our main result we answer this question for \emph{chains}, which are the simplest non-trivial class of compressed trees defined in Section~\ref{sec:chains}.
In Theorem~\ref{theo:decompressed}, we prove that the expected decompressed size of chains of size $n$ chosen uniformly at random is asymptotically equivalent to 
\[
	\frac{C \, e^{c \sqrt{n}}}{n^{1/4}}
\]
for explicit constants $C, c >0$.
The main difficulty is that it is much more difficult to enumerate DAGs than trees.

\begin{figure}[t]
	\centering
	\includegraphics[width=1\textwidth]{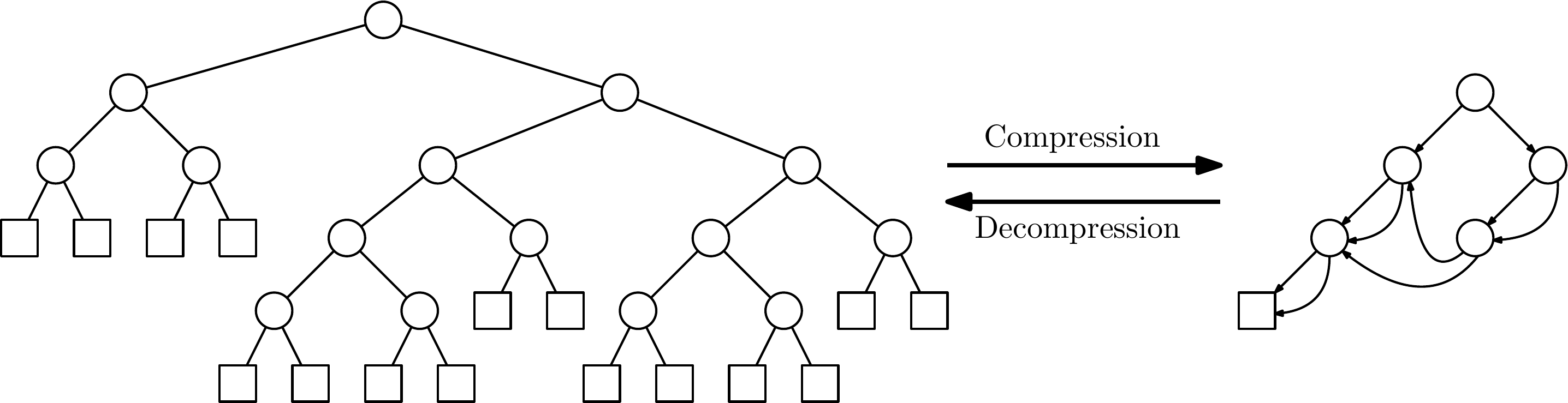}
	\caption{DAG compression: A binary tree $T$ of size $15$ on the left and its compressed representation as DAG $G$ of compressed size $5$ on the right, in which repeated occurrences of the same subtrees are factored. Thus, the \emph{decompressed tree size} of $G$ is $|G|_{\Tc} = 15$; see Definitions~\ref{def:pointersetc}, \ref{def:decompression}, and \ref{def:decompressedtreesize}. }
	\label{fig:compression}
\end{figure}

In~\cite{GenitriniGittenbergerKauersWallner2016} we considered the reverse question of enumerating all compressed\footnote{Previously, the term \emph{compacted tree} was used.} binary trees of compressed size $n$. 
We showed that after bounding one parameter related to the height, the associated generating functions are D-finite\footnote{A function is \emph{D-finite} if it satisfies a linear differential equation with polynomial coefficients.}. 
Later, in~\cite{ElveyPriceFangWallner2019Compacted} we solved the counting problem of unconstrained compressed binary trees using an entirely different approach on bivariate recurrence relations. 
We showed that their number grows like 
\[
	\Theta\left( n! \, 4^n\, e^{3a_1 n^{1/3}} \, n^{3/4} \right)
\]
for $n \to \infty$, where $a_1 \approx -2.331$ is the largest root of the Airy function $\Ai(x)$ of the first kind\footnote{The Airy function $\Ai(x)$ of the first kind is characterized by $\Ai''(x) = x \Ai(x)$ and $\lim_{x \to \infty} \Ai(x) = 0$.}. 
Such a stretched exponential term did not appear in the bounded case, proving a phase transition between the bounded and unbounded model. 
Similar phenomena have recently also been shown in the number of relaxed $k$-ary trees~\cite{GhoshDastidarWallner2024Relaxed},
minimal deterministic finite automata accepting a finite binary language~\cite{ElveyPriceFangWallner2020DFA},
Young tableaux with walls~\cite{BanderierWallner2021Walls},
and phylogenetic tree-child networks~\cite{Changetal2024dcombining,Fuchs2021TreeChild}.
\subsection{Main result}
In this paper we refine the enumeration of compressed trees and, for the first time, analyze the distribution of the decompressed sizes of compressed trees of size $n$ chosen uniformly at random.
Our main result is Theorem~\ref{theo:decompressed}, which gives the asymptotic expected decompressed size for the simplest non-trivial class of compressed trees, so-called chains defined in Section~\ref{sec:chains}.
In it we show that another stretched exponential appears for chains of arbitrary fixed out-degree $k \geq 2$. 

\subsection{Plan of the paper}
In Section~\ref{sec:decompressedtreesize} we introduce the decompression algorithm and the key concept of this paper: the decompressed tree size, which lifts a tree statistic to a DAG statistic.
In Section~\ref{sec:chains} we define the DAG class of chains and state our main result about their decompression.
Then, we introduce in Section~\ref{sec:cgf} $d$-exponential generating functions, which are a natural generalization of exponential generating functions, and use them in Section~\ref{sec:rightheight} to prove our main result. 
In Section~\ref{sec:combinatorics} we discuss and study the rich combinatorics of the binary case with links to Laguerre polynomials and permutations, explaining the remarkably simple closed forms.
In Section~\ref{sec:additionchains} we discuss connections with addition chains. 
Finally, in Section~\ref{sec:conclusion} we present several further research directions.

\section{The decompressed tree size of rooted DAGs}
\label{sec:decompressedtreesize}

\begin{figure}[t]
	\centering
	\includegraphics[width=1\textwidth]{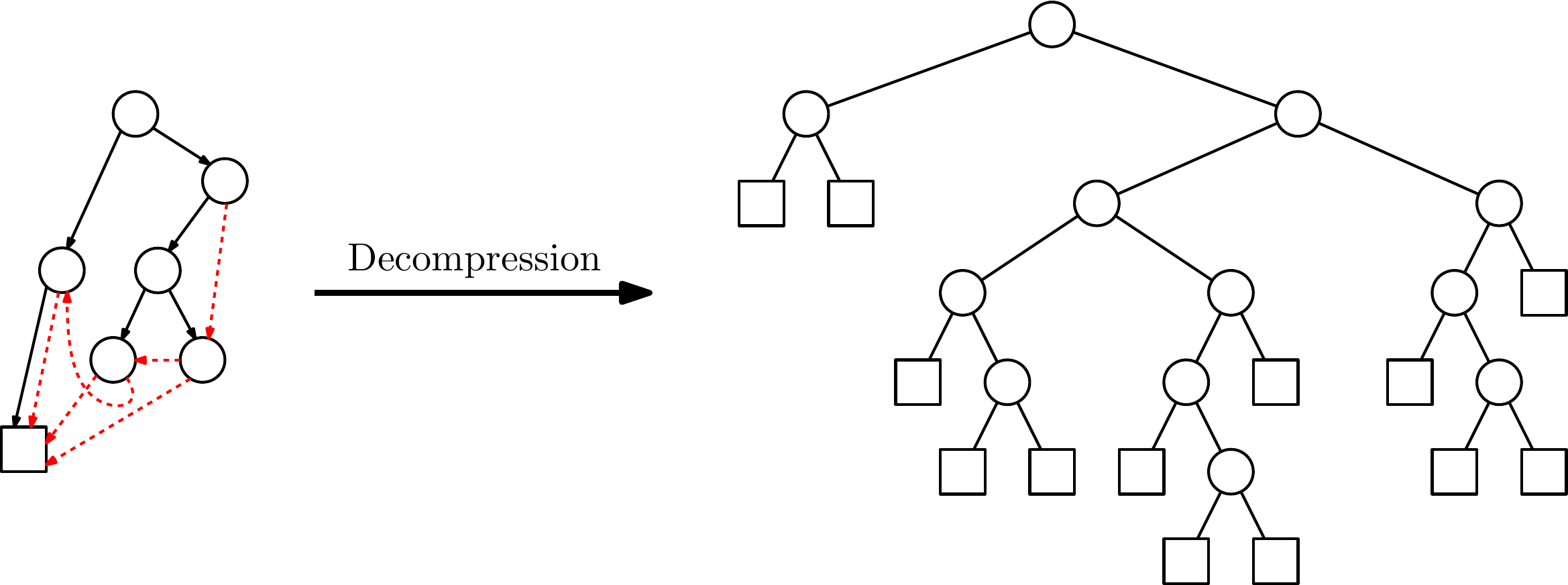}
	\caption{A rooted ordered binary DAG $G$ with 6 internal nodes (left) and its associated decompressed binary tree $D(G)$ with 13 internal nodes (right). 
	In ordered DAGs the outgoing edges have a left-to-right order. 
	In $G$ the internal edges are shown as solid black lines and  pointers as dashed red lines. Therefore, the decompressed tree size of $G$ is $|G|_{\Tc} = 13$; see Definitions~\ref{def:pointersetc}, \ref{def:decompression}, and \ref{def:decompressedtreesize}.}
	\label{fig:decompression}
\end{figure}

The motivation for this work is to understand the compression of rooted trees. 
Note that after compression, the root becomes the unique source of the DAG.
For this purpose we consider only such DAGs. 
To stress the connection with trees, we henceforth call DAGs with a unique source \emph{rooted DAGs}.
In the language of graph theory, this is equivalent to the existence of a unique subgraph that is an arborescence.
Note that all concepts can be directly generalized to DAGs with multiple roots, at the cost of losing the connection to trees. 

The key concept of this paper is the decompressed 
size of a rooted DAG.
For this purpose we now define the decompression algorithm, which associates with each rooted DAG a rooted tree. 
We start with a few definitions.

\begin{definition}[Parameters of rooted DAGs]
	\label{def:pointersetc}
	The \emph{spine} of a rooted DAG $G$ is its unique spanning tree computed in a depth-first search. 
	The edges of the spine are referred to as \emph{internal edges} and the other edges of $G$ as \emph{external edges} or \emph{pointers}.
	The \emph{fringe subgraph} of a node $v \in G$ is the induced subgraph of all nodes that can be reached from $v$. 
\end{definition}

Now we are ready to define the {decompression operator $\DD$} that assigns to each rooted DAG a rooted tree that we call its decompressed tree.

\begin{definition}[Decompressed tree]
\label{def:decompression}
Let $G$ be a rooted DAG. Its \emph{decompressed tree $\DD(G)$} is defined as follows:
\begin{enumerate}
	\item Replace the sinks by leaves, i.e., nodes with out-degree $0$.
	\item Traverse the DAG along internal edges in postorder. 
	\item Iteratively, once all children of a node~$v$ have been processed, replace all pointers emanating from $v$ by edges to copies of the fringe subtrees they are pointing to. 
\end{enumerate}
\end{definition}

\begin{remark}
Note that different DAGs $G_1 \neq G_2$ might have the same decompressed trees, i.e., $D(G_1)=D(G_2)$.
However, for each tree $T$, there is a unique rooted DAG $G$ with minimal number of nodes, such that $D(G)=T$.
In the minimal DAG all fringe subgraphs are unique.
We will not explore this uniqueness, since we are here interested in the decompressed properties of DAG classes, rather than the compression of tree classes. 
For more on the topic of uniqueness in trees and DAGs, we refer to~\cite{flss90,GenitriniGittenbergerKauersWallner2016,ElveyPriceFangWallner2019Compacted,bousquet2015xml,RalaivaosaonaWagner2015Repeated}.
\end{remark}

The decompression operator allows us now to introduce the new key statistic of this paper for DAGs: the decompressed tree size; see Figure~\ref{fig:decompression}.

\begin{definition}[Decompressed tree size]
\label{def:decompressedtreesize}
	For a given size function $|\cdot| : \Tc \mapsto \N$ on trees $\Tc$,
	the \emph{decompressed tree size} $|G|_{\Tc}$ of a DAG $G$ is defined as 
	\[|G|_{\Tc} := |\DD(G)|.\]
\end{definition}

The type of the decompressed tree depends on the chosen family of DAGs. 
In this paper we will study rooted ordered (also called plane) DAGs, in which the children of each node have a distinguished order and we will use as our size function $|\cdot|$ the number of internal nodes.
Alternatively, the total number of nodes, or the number of leaves are also possible, yet these are less natural for the subsequent results and lead to more complicated formulas.

\begin{remark}[Decompressed tree statistics for DAGs]
	Note that the decompressed tree size is just one example of a tree statistic that can be used to define a DAG statistic using the decompression operator.
	Other interesting choices are for example the height or the width of the associated tree.
	This is particularly interesting, as trees are well-studied~\cite{drmo09} and very closely related to probabilistic objects such as Brownian excursions~\cite{jans07}.
	The decompression framework allows lifting all these concepts into the world of DAGs.
\end{remark}

\section{The expected decompressed size of chains}
\label{sec:chains}

We will now introduce the family of DAGs that we consider: $k$-ary chains.
Let $k$ be a positive integer.
We define \emph{$k$-ary DAGs} as rooted ordered DAGs in which each node has exactly $k$ outgoing edges.
The corresponding decompressed trees are classical rooted ordered $k$-ary trees, enumerated by the Fuss--Catalan numbers $\frac{1}{(k-1)n+1} \binom{kn}{n}$.
The most famous subclass are binary trees enumerated by the ubiquitous Catalan numbers $\frac{1}{n+1} \binom{2n}{n}$; see, e.g.,~\cite{stanley2015catalan} for links to hundreds of combinatorial objects.

In addition to the out-degrees, one may also fix the allowed shapes.
We consider fixed shapes of the corresponding spines. 
Such families were (asymptotically) enumerated in~\cite{GenitriniGittenbergerKauersWallner2016} for binary DAGs with spines of fixed right height.
The simplest non-trivial example are \emph{chains} that are DAGs conditioned to have spines of right height $0$, i.e., the spine is a sequence of nodes; see Figure~\ref{fig:binarychain}. 
All fringe subgraphs of these DAGs are unique and thus all its decompressed trees are distinct.

\begin{figure}[h!]
	\centering
	\includegraphics[width=1\textwidth]{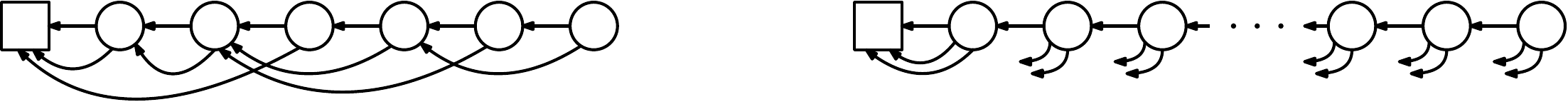}
	\caption{Left: A binary chain (i.e., $k=2$) with $6$ internal nodes. The unique root is on the very right and the unique sink on the very left.
	Right: The shape of ternary chains. 
	}
	\label{fig:binarychain}
\end{figure}

Let us briefly describe, why this class is interesting.
First, since pointers can point to any node to the left, the number of $k$-ary chains with $n$ internal nodes is $(n!)^{k-1}$.
Hence, the cardinality of this subclass grows super-exponentially in the number~$n$ of internal nodes. 
Second, the study of this subclass was the key in all previous papers to derive asymptotic results.
Third, among all compressed trees of size $n$, this class includes the smallest and the largest possible decompressed tree:
If all pointers end in the sink, the decompressed tree has size $n$, 
while if all pointers connect with the node just before in postorder the decompressed tree has size $k^n-1$. 

\medskip

Formally, let $k \geq 2$ be an integer and let $\CC$ be the set of $k$-ary chains. 
We denote by $\CC_n = \{ C \in \CC ~:~ |C| = n \}$ the set of $k$-ary chains conditioned to have $n$ internal nodes and by $c_n :=  |\CC_n| = (n!)^{k-1}$ its cardinality.
Let $X_n$ be the random variable associated with the decompressed size of chains $\CC_n$ chosen uniformly at random, i.e.,
\begin{align*}
	\PR( X_n = k ) = \frac{ \left| \{ C \in \CC_n ~:~ |C|_{\Tc} = k \} \right| }{c_n}.
\end{align*}
We are interested in the distribution of $X_n$. 
Our main result is the following asymptotics of the average decompressed size $\E(X_n)$ of $k$-ary chains which includes a stretched exponential.
For comparison, the minimal size is $n = e^{\log (n)}$  (i.e., all pointers to sink)
and the maximal size is $k^n-1 \approx e^{n \log(k)}$ (i.e., all pointers to previous internal node).

\begin{theorem} 
	\label{theo:decompressed}
	The expected decompressed tree size of $k$-ary chains with $n$ internal nodes is for $n\to \infty$ asymptotically given by
	\begin{align*}
		\E(X_n) & 	= \frac{1}{2\sqrt{e^{k-1} \pi}  \, (k-1)^{5/4}} \frac{e^{2 \sqrt{(k-1) n}}}{n^{1/4}} \left(1 + \LandauO\left(\frac{1}{\sqrt{n}}\right)\right).
	\end{align*}
\end{theorem}

The next two sections are devoted to the proof of this theorem. 
In Section~\ref{sec:cgf} we introduce a special class of generating functions that helps to model the problem,
and in Section~\ref{sec:rightheight} we study the depth of a node in the decompressed tree that will then lead to the final proof.

\section{The \texorpdfstring{$d$}{d}-exponential generating function}
\label{sec:cgf}

We will use a special class of generating functions. 
Let $d \geq 0$ be an integer.
For a counting sequence $(f_n)_{n \geq 0}$ we define the \emph{$d$-exponential generating function}
\begin{align}
	\label{eq:cgf}
	F(z) &= \sum_{n \geq 0} f_n \frac{z^n}{(n!)^d}.
\end{align}
For $d=0$ we recover ordinary generating functions and for $d=1$ exponential generating functions.
In the case of $k$-ary chains we will use $d=k-1$, as in this case the normalization is exactly the number $c_n=|\CC_n|$ of chains.
The choice in~\eqref{eq:cgf} allows us to extend the symbolic calculus on exponential generating functions introduced in~\cite{GenitriniGittenbergerKauersWallner2016} to $k$-ary DAGs. 
The subsequent lemma generalizes~\cite[Lemma~6.1]{GenitriniGittenbergerKauersWallner2016}, which concerns the symbolic construction of appending a new root node, and is the key symbolic construction.

\begin{figure}[h!]
	\centering
	\includegraphics[width=0.4\textwidth]{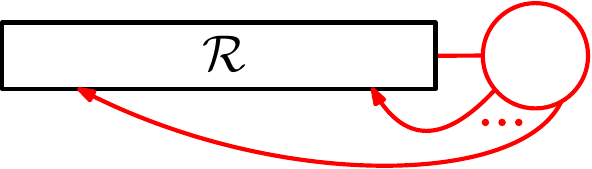}
	\caption{Adding a new root and $k-1$ pointers (in red) to a rooted $k$-ary DAG $\Rc$; see Lemma~\ref{lem:addingaroot}.
	}
	\label{fig:newroot}
\end{figure}

\begin{lemma}[Adding a new root]
\label{lem:addingaroot}
Let $\Rc$ be a combinatorial subclass of rooted $k$-ary DAGs.
Let~$\Sc$ be the combinatorial class whose elements consist of a new root node with out-degree~$k$, with an element of $\Rc$ as first child, and with $k-1$ pointers as children $2,\dots,k$; see Figure~\ref{fig:newroot}.   
Then, the associated $(k-1)$-exponential generating functions satisfy
\begin{align*}
	S(z) = z R(z).
\end{align*}
\end{lemma}

\begin{proof}
	Consider a DAG $R \in \Rc$ of size $n$. 
	First, we deterministically attach a new node as first child to the root of $R$.
	Then, we attach $k-1$ additional pointers to this new root, where each of these may point to one of the $n+1$ possible nodes in $R$ ($n$ internal nodes and the leaf).
	
	Now, let $$R(z) = \sum_{n \geq 0} r_n \frac{z^n}{(n!)^{k-1}}$$ be the $(k-1)$-exponential generating function of $\Rc$.
	Then, a direct computation shows
	\begin{align*}
		S(z) &= \sum_{n \geq 0} (n+1)^{k-1} r_n \frac{z^{n+1}}{((n+1)!)^{k-1}} \\
		     &= \sum_{n \geq 0} r_n \frac{z^{n+1}}{(n!)^{k-1}}
					= z R(z).
	\end{align*}
	This proves the claim.
\end{proof}

Now, Lemma~\ref{lem:addingaroot} directly implies that the generating function of $k$-ary chains has a very simple closed form.
This generalizes~\cite[Corollary~6.2]{GenitriniGittenbergerKauersWallner2016} to out-degree $k$, always leading to the \emph{same} closed form yet different rescalings.

\begin{corollary}
	\label{cor:chains}
	The $(k-1)$-exponential generating function 
	\begin{align*}
		C(z) = \sum\limits_{C \in \CC} c_n \frac{z^{|C|}}{(n!)^{k-1}}
	\end{align*}	
	of $k$-ary chains $\CC$ is
	\begin{align*}
		C(z) = \frac{1}{1-z}.
	\end{align*}
\end{corollary}

\begin{proof}
	The proof follows directly from the chosen construction of the $(k-1)$-exponential generating function in~\eqref{eq:cgf} since $f_n = (n!)^{k-1}$. 
	
	As it will be instrumental in the remainder, we will now give an alternative derivation using solely symbolic constructions from Lemma~\ref{lem:addingaroot}.
	Note that the proof follows verbatim the proof of~\cite[Corollary~6.2]{GenitriniGittenbergerKauersWallner2016}.
	Symbolically, a chain is either just a leaf of size $0$, or it is constructed from a non-empty chain by appending a new root node.
	On the level of generating functions, this translates into the functional equation
	\begin{align*}
		C(z) &= 1 + z C(z),
	\end{align*}
	and the claim follows.
\end{proof}

In the following proposition we define the two most important constructions of this paper.
First, we generalize the construction of attaching a single root node to the case of an arbitrary sequence of root nodes; see Figure~\ref{fig:seqrootskary}.
Second, we give the following DAG-specific interpretation of the classical pointing operator and its inverse; see Figure~\ref{fig:add-remove-firspointer}.
Note that the last interpretations are independent of the rescaling by $(n!)^{d}$ in the generating function~\eqref{eq:cgf}.
Its proof follows verbatim the one of \cite[Proposition~6.4]{GenitriniGittenbergerKauersWallner2016} which is why we omit its repetition here.

\begin{figure}[h!]
	\centering
	\includegraphics[width=0.95\textwidth]{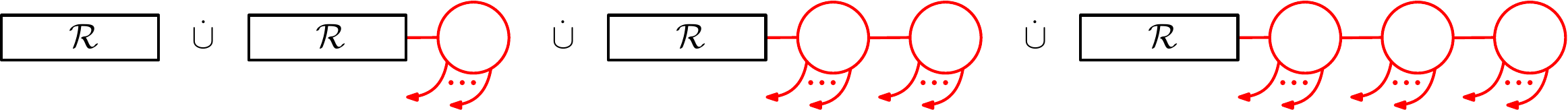}
	\caption{The sequence construction of repeatedly adding roots (in red) from Proposition~\ref{prop:remove-add-pointer}. Note that the first case corresponds to adding no new roots.
	}
	\label{fig:seqrootskary}
\end{figure}

\begin{proposition}[{{\cite[Proposition~6.4]{GenitriniGittenbergerKauersWallner2016}}}]
	\label{prop:remove-add-pointer}
	Let $\Rc$ be a combinatorial subclass of rooted $k$-ary DAGs.
	The $(k-1)$-exponential generating function $S(z)$ corresponding to the class obtained by appending an arbitrary (possibly empty but finite) sequence of nodes with out-degree $k$ to the root of $\Rc$ shown in Figure~\ref{fig:seqrootskary} is given by
	\begin{align*}
		S(z) &= \frac{1}{1-z} R(z).
	\end{align*}
	Let $\Rc_+$ be the class derived from $\Rc$ by adding a new first pointer to the root
	and let $\Rc_-$ be the class derived from $\Rc$ by removing the first pointer of the root; see Figure~\ref{fig:add-remove-firspointer}.
	Then, the corresponding $(k-1)$-exponential generating functions are
	\begin{align*}
		R_+(z) &= z \frac{d}{dz}R(z) + R(0), \\
		R_-(z) &= \int \frac{R(z)-R(0)}{z} \, dz.
	\end{align*}
\end{proposition}

\begin{figure}[h!]
	\centering
	\includegraphics[width=0.7\textwidth]{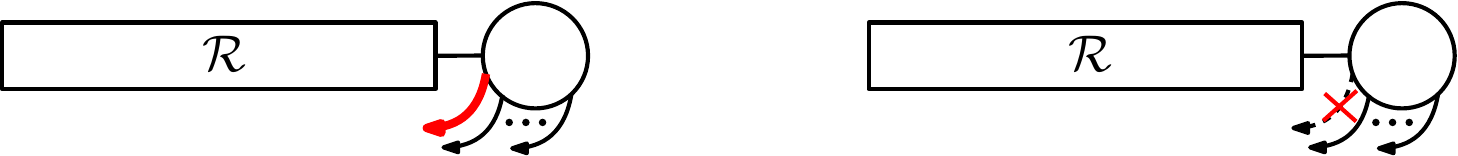}
	\caption{Left: The class $\Rc_+$ obtained by adding a new pointer (red bold edge) to the root of $\Rc$. Right: The class $\Rc_-$ obtained by removing the first pointer of the root of $\Rc$. Note that the resulting graphs are in general not $k$-ary DAGs since the root has too many or too few pointers, yet these constructions are needed in the recursive construction of Theorem~\ref{theo:Dellclosedform}.
	}
	\label{fig:add-remove-firspointer}
\end{figure}

\begin{remark}[Comparison with the binary case]
	The result of Proposition~\ref{prop:remove-add-pointer} on adding and removing pointers is the same as in the binary case, with the only extension that the location of the added pointer needs to be deterministic. 
	Hence, the claim stays valid if a different pointer is removed/added at a fixed location, e.g., the second, third, last, etc.
\end{remark}

\section{The compression depth}
\label{sec:rightheight}

The key idea to prove Theorem~\ref{theo:decompressed} is to relate the nodes in the decompressed trees with a statistic in the corresponding DAGs; see Figure~\ref{fig:compressiondepth}.

\begin{lemma}
	\label{lem:Rtrajectories}	
	Let a rooted DAG $G$ be given.
	Each node $v$ in $\DD(G)$ is in bijection with a path~$P_v$ in $G$ from the root to the compressed instance of $v$.
\end{lemma}

\begin{proof}
	Let $G$ be a graph and $T=\DD(G)$ be its decompressed tree.
	Observe that by definition the decompression operator $\DD$ does not change the local topologies of fringe subgraphs. 
	Therefore, each path in $G$ is by the decompression operator in bijection with a path in~$T$, and vice versa.
	By definition, each node in a tree has a unique path to its root. 
	Hence, we can associate each node in~$T$ with this unique path in~$T$, and the claim follows. 
\end{proof}

Note that in $G$ several paths might end in the same node, but different paths correspond to different nodes in the decompressed tree $\DD(G)$.
From now on, we use the notation $P_v$ for the unique path in $G$ that is associated with a node $v \in \DD(G)$.
In a next step, we use these paths to define the compression depth and the compression level of a node using the concept of pointers from Definition~\ref{def:pointersetc}.
Note that we introduce the redundant concept of a compression level because most formulas are simpler when expressed in this way.  

\begin{definition}[Compression depth and level]
\label{def:compressiondepth}
Let $G$ be a rooted DAG.
The \emph{compression depth} of a node $v \in \DD(G)$ is defined as the number of pointers in $P_v$.
The \emph{compression level} of a node $v$ is defined as its compression depth plus one; see Figure~\ref{fig:compressiondepth}.
\end{definition}

\begin{figure}[t]
	\centering
	\includegraphics[width=1\textwidth]{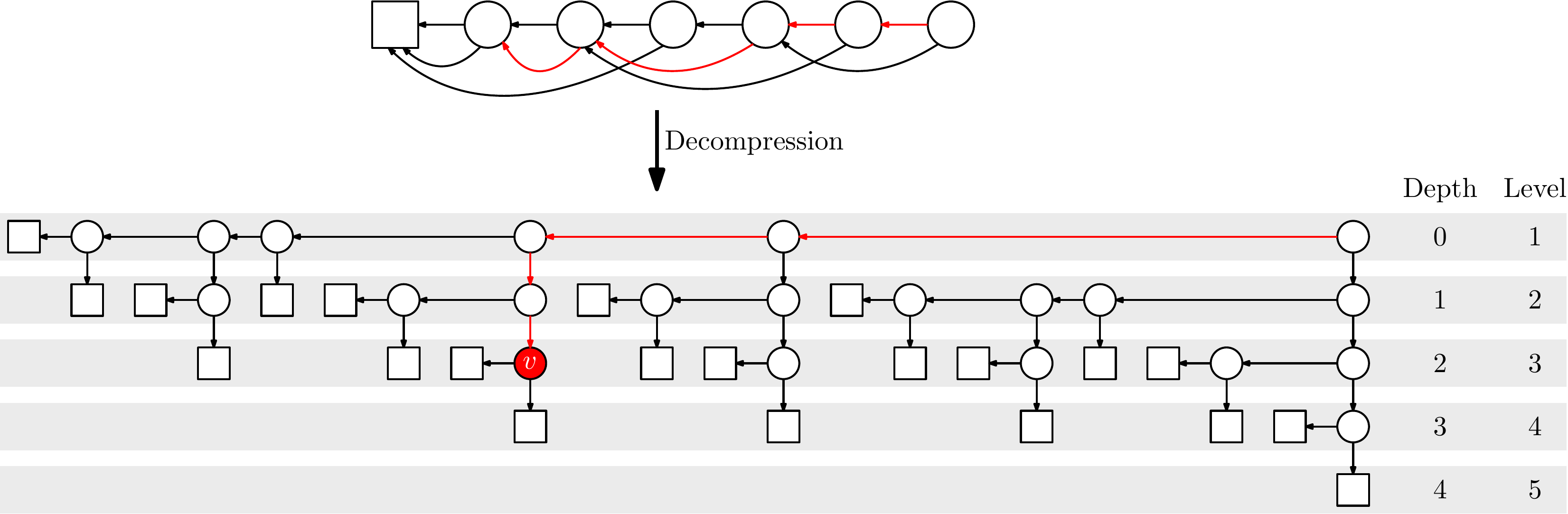}
	\caption{A binary chain $C$ (top) and its decompressed tree $\DD(C)$ (bottom). 
	The decompression depths and levels from Definition~\ref{def:compressiondepth} are stated on the right. 
	Lemma~\ref{lem:Rtrajectories} associates the red path $P_v$ in $C$ with the node $v \in \DD(C)$. }
	\label{fig:compressiondepth}
\end{figure}

\begin{remark}[Compression depth as a measure of redundancy]
The compression depth measures how often the compression operator was used to compress a specific node.
In this sense, it is a qualitative measure for the redundancy of this specific node in the tree.
Note that this notion depends strongly on the chosen 
traversal in the decompression operator. 
Other traversals might assign other compression levels to the same nodes.
\end{remark}

Our strategy is to count the number of decompressed nodes per compression level 
 in all chains of size $n$ using the previously introduced $(k-1)$-exponential generating functions.
For this purpose, let $d_{\ell,n}$ be the number of nodes on compression level $\ell \geq 1$ in $\DD(\CC_n)$, i.e., all decompressed $k$-ary chains of size~$n$.
We denote the associated $(k-1)$-exponential generating function by 
\begin{align*}
	D_{\ell}(z) = \sum_{n \geq 0} d_{\ell,n} \frac{z^n}{(n!)^{k-1}}. 
\end{align*}
In the following result we show that it has a remarkably simple closed form.

\begin{theorem}
	\label{theo:Dellclosedform}
	The $(k-1)$-exponential generating function $D_{\ell}(z)$ of the number of nodes on compression level $\ell \geq 1$ in all decompressed $k$-ary chains is given by
	\begin{align*}
		D_{\ell}(z) 
			&= \frac{(k-1)^{\ell-1} z^{\ell}}{\ell! (1-z)^{\ell+1}}%
			= \frac{(k-1)^{\ell-1} }{\ell!} \sum_{n \geq \ell} \binom{n}{\ell} z^{n}%
		.
	\end{align*}
	Therefore, the bivariate generating function 
	$
		D(z,q) = \sum_{\ell \geq 1} D_{\ell}(z) q^{\ell}
	$ 
	is given by
	\begin{align}
		\label{eq:Dzqclosedform}
		D(z,q) &= \frac{1}{(k-1)(1-z)} \left( e^{\frac{(k-1)qz}{1-z}} - 1 \right).
	\end{align}
\end{theorem}

\begin{proof}
	We will compute the generating functions $D_{\ell}(z)$ inductively with respect to the level~$\ell$. 
	By Lemma~\ref{lem:Rtrajectories} a node is on level $\ell$ (i.e., has compression depth $\ell-1$) if and only if it is reached from the root after traversing exactly $\ell-1$ pointers. 
	
	For level $\ell=1$ we consider all nodes at compression depth $0$. 
	Each node that can be reached without traversing any pointers lies on this level.
	Therefore, these are all nodes of the spine.
	Thus, $d_{1,n} = n \cdot (n!)^{k-1}$ and we directly get the claimed closed form $D_1(z) = \frac{z}{(1-z)^2}$.
	
	
	For $\ell \geq 2$ we continue by induction. 
	Assume that the generating function $D_{\ell}(z)$ satisfies the claimed closed form.
	In order to compute the nodes on level $\ell+1$, we lift the nodes on level $\ell$ by one level.
	The idea behind this is the following generic structure:
	There are nodes $v$ and $w$ connected by a pointer such that $v$ is the root of $D_{\ell}(z)$ and there is a (possibly empty) sequence of nodes between $v$ and $w$, as well as before $w$.
	We can model this symbolically as follows:
	\begin{enumerate}
		\item attach a (possibly empty) sequence of nodes to $D_{\ell}(z)$; 
		\item\label{item:newrootw} attach a node $w$ as new root;
		\item\label{item:replace} replace one of the $k-1$ pointers of $w$ by a pointer to $v$;
		\item attach another sequence of (possibly empty) nodes to $w$. 
	\end{enumerate}
	Using Proposition~\ref{prop:remove-add-pointer}, the above operations correspond to the following operations on generating functions, respectively:
	\begin{enumerate}
		\item multiply by $\frac{1}{1-z}$;
		\item multiply by $z$;
		\item divide by $z$ and integrate with respect to $z$, then multiply by $k-1$ since we can replace any of the $k-1$ pointers and not only the first one (replacing the second, third, etc. pointer of the root is analogous to the first one from Proposition~\ref{prop:remove-add-pointer});
		\item multiply by $\frac{1}{1-z}$.
	\end{enumerate}
	Therefore, we get by induction:
	\begin{align*}
		D_{\ell+1}(z) &= \frac{k-1}{1-z} \int \frac{1}{1-z} D_{\ell}(z) \, dz 
		                   = \frac{(k-1)^{\ell} z^{\ell+1}}{(\ell+1)! (1-z)^{\ell+2}}.
	\end{align*}
	Note that $D_{\ell}(0)=0$ for $\ell \geq 1$.
	The coefficients then follow directly from this representation by the generalized binomial theorem.
	
	Finally, the closed form of $D(z,q)$ follows directly by summation over $\ell \geq 1$.
\end{proof}

Using this closed form, we are now ready to prove our main result on the average decompressed size of an element of $\CC_n$ when chosen uniformly at random.
Remarkably, we see that a stretched exponential appears.

\begin{proof}[Proof of Theorem~\ref{theo:decompressed}]
	%
	First observe that $\E(X_n) = [z^n]D(z,1)$, as the rescaling by $(n!)^{k-1}$ is equal to the number of elements in $\CC_n$. 
	Then, this result follows directly from an application of the saddle-point method~\cite[Chapter~VIII]{flaj09} to 
	\begin{align*}
		D(z,1) = \frac{1}{(k-1)(1-z)} \left( e^{\frac{(k-1)z}{1-z}} - 1 \right).
	\end{align*}
	For this purpose, we use Cauchy's integral formula, which expresses the coefficients of the power series in terms of a complex contour integral.
	Subsequently, it is possible to transform this contour such that the main contribution can be approximated by a Gaussian integral. 
	We omit the standard technical details of proving that all neglected terms are negligible that arise from pruning and bounding the tails of the integral.
	See~\cite[pp.~596]{flaj09} for this concrete example for $k=2$.
\end{proof}

\begin{remark}
	An alternative approach to derive the previous result relies on a recurrence for $\E(X_n)$ that can be derived similarly to the recurrence that determines the average case complexity of quicksort.
	In the binary case, this approach leads to a recurrence analyzed by Kac~\cite{Kac1989Ulam}, who proves the stretched exponential term $e^{2\sqrt{n}}$, but does not give the polynomial term $n^{-1/4}$ nor the multiplicative constant.
\end{remark}

In the next section, we will discuss the rich combinatorics behind the closed form~\eqref{eq:Dzqclosedform}.

\section{Combinatorics of decompressed chains}
\label{sec:combinatorics}

We first discuss the binary case, i.e., $k=2$.
Let us start by taking a closer look at the underlying counting sequences of $D(z,1)$, which correspond to the total number of decompressed nodes in binary chains $\CC$. 
The exponential generating function is given by
\begin{align*}
	D(z) &:= \frac{1}{1-z} \left( e^{\frac{z}{1-z}} - 1 \right) 
			\\&
			= z + 5 \frac{z^2}{2!} + 28 \frac{z^3}{3!} + 185 \frac{z^4}{4!} + 1426 \frac{z^5}{5!} + 12607 \frac{z^6}{6!} + \dots.
\end{align*}
This counting sequence is found in the OEIS\footnote{The On-Line Encyclopedia of Integer Sequences \url{https://oeis.org}} as \OEIS{A070779}, which enumerates strictly partial permutations of~$[n] := \{1,\dots,n\}$.
\emph{Partial permutations} are mappings $\sigma$ from $[n]$ to $[n]^* := [n] \cup \{\infty\}$ such that they are injections when restricted to $\sigma^{-1}([n])$.
Equivalently, they are words of length~$n$ with letters from $[n]^*$ in which each $i \in [n]$ may appear at most once.
\emph{Strictly partial permutations} of~$[n]$ are partial permutations that are not permutations, i.e., at least one $\infty$ appears.
Therefore, the nodes in decompressed binary chains $\DD(\CC)$  are in bijection with strictly partial permutations, while binary chains $\CC$ are in bijection with permutations.
Moreover, using permutation matrices, one sees that rook placements are in bijection with permutations, and partial rook placements with partial permutations.
This notion also allows us to give $D(z,q) = \frac{1}{1-z} ( e^{\frac{qz}{1-z}} - 1)$ a combinatorial interpretation.
Observe that Laguerre polynomials $L_n(q)$ are defined by
	\begin{align*}
		\sum_{n \geq 0} L_{n}(q) z^n = \frac{1}{1-z} e^{- \frac{qz}{1-z}}.
	\end{align*}
	Therefore, on the one hand, $n!(L_{n}(-q) -1)$ are the generating polynomials of the distribution of nodes per layer in $\DD(\CC_n)$.
	On the other hand, $n! x^n L_{n}(-x^{-1})$ are known as \emph{rook polynomials}, in which the coefficient of $x^k$ is equal to the number of non-attacking configurations of $k$ rooks on a $n \times n$ chessboard.
	
We will now discuss the combinatorics of the remarkably simple formula for $d_{\ell,n}$ from Theorem~\ref{theo:Dellclosedform} and give an alternative proof using a bijective argument.
In the binary case, the corresponding sequences in the OEIS for $\ell=1,\dots,5$ are 
\OEISs{A001563},
\OEISs{A001809},
\OEISs{A001810},
\OEISs{A001811},
\OEISs{A001812}, respectively.


\begin{proposition}
	\label{prop:delln}
	The number $d_{\ell,n}$ of internal nodes on compression level $\ell$ in all decompressed $k$-chains $\DD(\CC_n)$ with $n$ internal nodes is
	\begin{align*}
		d_{\ell,n} &= n! \frac{(k-1)^{\ell-1}}{\ell!} \binom{n}{\ell}.
	\end{align*}
\end{proposition}


\begin{proof}
	We will prove that $\ell! d_{\ell,n} = n! (k-1)^{\ell-1} \binom{n}{\ell}$, by showing that both sides count the same objects. 
	On the one hand, we interpret the left-hand side canonically as the number of nodes on level $\ell$ in $\DD(\CC_n)$ where each node comes with multiplicity $\ell!$.
	On the other hand, we now describe a direct construction that marks each element on level $\ell$ exactly $\ell!$ times and that is enumerated by the right-hand side.	
	
	We will again use the characterization given by Lemma~\ref{lem:Rtrajectories}:
	For a given $C \in \CC_n$, each decompressed node $v_0 \in \DD(C)$ on level $\ell$ is in bijection with a path $P_{v_0}$ in $C$ consisting of $\ell-1$ pointers.
	The key observation is now that $P_{v_0}$ can be uniquely identified by the final node $v_0$, the nodes $V_P = \{v_1,\dots,v_{\ell-1} \}$, and numbers $w_i \in \{1,2,\dots,k-1\}$ for $i=1,\dots,\ell_1$ as follows:
	The path starts at the root and traverses $C$ until it reaches~$v_0$ such that it continues along the $w_i$-th pointer at the node $v_i$ and along an internal edge otherwise. 
	Our strategy is now to choose $\ell$ internal nodes of $C$ and check whether these define such a \emph{valid traversal}.
	
	For a given $C \in \CC_n$ we label the nodes in the spine increasingly starting from the sink with $\{0,1,\dots,n\}$.
	From now on we will identify the nodes with their labels.
	Then, we choose $\ell$ of its $n$ internal nodes, say $v_0,\dots,v_{\ell-1}$ such that 
	\[
		v_0 < v_1 < \dots < v_{\ell-1}.
	\]
	Let $p_1,\dots,p_{\ell-1}$ be the targets of the pointers associated with the nodes $v_1,\dots,v_{\ell-1}$, respectively. 
	Here, we have for each node $k-1$ choices,	which correspond to the numbers $w_i$ for $i=1,\dots,\ell-1$.
	Note that in the following construction we will not change the positions $w_i$ of the pointers. 
	This independence will then result in the factor $(k-1)^{\ell-1}$ in the final formula.
	
	Now observe that by definition of $\CC$, we have $p_i < v_i$. 
	In order to be a valid traversal, these must additionally satisfy 
	\begin{align}
		\label{eq:validvipi}
		v_{i-1} \leq p_{i}.
	\end{align}
	In other words, the nodes and their pointers must interlace (with allowed overlaps):
	\begin{align*}
		v_0 \leq p_1 < v_1 \leq p_2 < v_2  \leq  \dots \leq p_{\ell-1} < v_{\ell-1}.
	\end{align*}
	

\begin{figure}[t]
	\centering
	\includegraphics[width=0.65\textwidth]{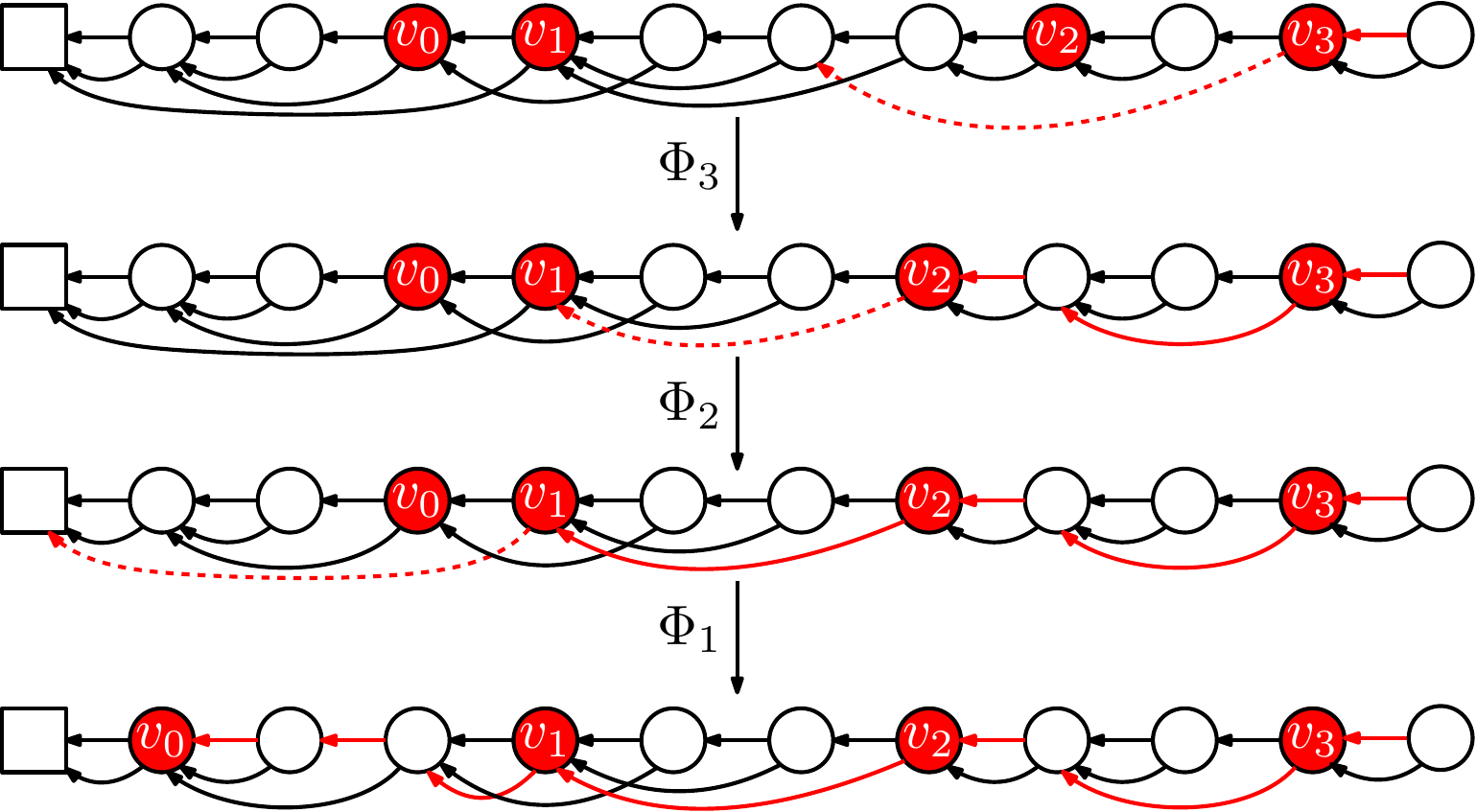}
	\caption{The mapping $\Phi = \Phi_1 \circ \Phi_2 \circ \Phi_3$ from the proof of Proposition~\ref{prop:delln}, mapping a binary chain with $4$ marked vertices to one that corresponds to a valid traversal $P_{v_0}$ shown in red. By Lemma~\ref{lem:Rtrajectories} $P_{v_0}$ is in bijection with a node in $D(C)$ at level $3$, since $P_{v_0}$ contains $3$ pointers.}
	\label{fig:binarychainPhiExample}
\end{figure}
			
	Let $T(\CC_n)$ be the class $\CC_n$ with $\ell$ marked internal nodes.
	We will now define a sequence of transformations \[\Phi_i : T(\CC_n) \to T(\CC_n)\] for $i=1,\dots,\ell-1$, where $\Phi_i$ maps the pointer $p_i$ to a valid one in the sense of~\eqref{eq:validvipi}. 
	An example is shown in Figure~\ref{fig:binarychainPhiExample}. 
	Let $\dot{C} \in T(\CC_n)$ be a compressed chain with $n$ internal nodes of which $\ell$ are marked (denoted by the dot).
	In order to define $\Phi_i$ we distinguish two cases depending on the $i$th pointer $p_i$ of $\dot{C}$.
	\begin{enumerate}
		\item If $p_i \geq v_{i-1}$, then $\Phi_i(\dot{C})=\dot{C}$, i.e., no change happens.
		\item If $p_i \in [v_{k-1},v_{k})$ for $k \leq i-1$ (where we set $v_{-1}=0$ to be the sink), then $\Phi_i(\dot{C})=\dot{C}'$, where $\dot{C}'$ is an identical copy of $\dot{C}$ except for the following changes: 
		In $\dot{C}'$ we denote the pointers and marked nodes by $p'_i$ and $v'_i$, respectively.		
		Let $d = v_{k}-p_i-1$ be the number of internal nodes between the target of $p_i$ and the next larger node $v_k$. 
		Then, we set
	\begin{enumerate}
		\item $p_i' = v_{i-1}$;
		\item $v_{j}' = v_{j}-d$ for $j=k,..., i-1$.
	\end{enumerate}
	In other words, we change the violating pointer and let it point to $v_{i-1}$, 
	and we move all marked nodes between $p_i$ and $v_{i}$ to the left (maintaining their local topology) until $v_{k}'=p_i+1$.
	\end{enumerate}
	Now, for any given $\dot{C}$, we first apply $\Phi_{\ell-1}$ correcting $p_{\ell-1}$, then $\Phi_{\ell-2}$ correcting $p_{\ell-2}$, etc. 
	Therefore, the mapping $\Phi = \Phi_1 \circ \Phi_2 \circ \dots \circ \Phi_{\ell-1}$ maps any traversal to a valid one.
	Note that for $i<j$ it is important to apply $\Phi_j$ before $\Phi_i$, since $\Phi_j$ might change the location of the marked nodes $v_{0},\dots,v_{j-1}$ but does not change $v_{j},\dots,v_{\ell-1}$.
	
	We will now prove that the mapping $\Phi$ is an $\ell!$-to-$1$ mapping between $T(\CC_n)$ and valid traversals.
	An example is shown in Figure~\ref{fig:PhiPreimage}.	
	To see this, observe that the mapping $\Phi_i$ for $p_i$ has the following $i+1$ distinct preimages:
	Either the element itself with a valid pointer $p_i$, or $i$ choices with an invalid pointer $p_i$.
	In the latter, the pointer $p_i$ can be placed in any of the intervals $I_k = [v_{k-1},v_{k})$ for $k=0,\dots,i-1$, by reversing the construction above as follows:
	We first compute $d=p_i-v_{i-1}$, which is the distance between $p_i$ and the next lower marked node $v_i$ (the traversal is valid!).
	Then we choose one of the $i$ intervals $I_k$ and change the pointer $p_i$ to point to the node $v_{k}-1$.
	Finally, we shift all marked nodes $v_{k},\dots,v_{i-1}$ $d$ steps to the right.
	Observe that all these shifts give a different element and do not change the marked nodes $v_{i},\dots,v_{\ell-1}$ nor any other pointers.
	Therefore, computing these preimages in the reversed order of $\Phi$, i.e., $i=1,2,\dots,\ell-1$, only the pointers of the nodes $V_P$ might change, but no other ones. 
	Now, as seen above the pointer $p_i$ has $i+1$ possible locations and hence any given valid traversal has exactly $\ell!$ different preimages.
\end{proof}

\begin{figure}[t]
	\centering
	\includegraphics[width=1\textwidth]{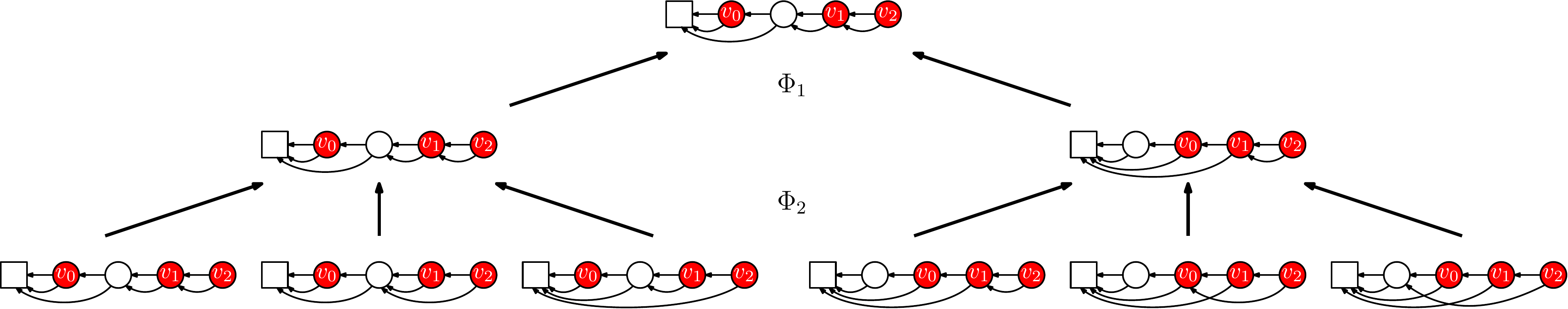}
	\caption{Example of all preimages of the valid traversal on the top with $3$ marked nodes under the mapping $\Phi = \Phi_1 \circ \Phi_2$ from the proof of Proposition~\ref{prop:delln}.
	For $\ell$ marked nodes the mapping $\Phi = \Phi_1 \circ \dots \circ \Phi_{\ell-1}$ is a $\ell!$-to-$1$ mapping.}
	\label{fig:PhiPreimage}
\end{figure}

\begin{remark}[Connection with inversion tables for binary chains]
	Note that each binary chain is in bijection with an inversion table~\cite{Wallner2023Inversion} (and therefore a permutation): 
	Label the nodes starting from the leaf by $0,1,\dots, n$.
	Then the associated inversion table $T = (t_1,t_2,\dots,t_n)$ is defined by $t_i = j$ if the pointer of node $i$ goes to node $j$. 
	By construction, $0 \leq t_i < i$, and each such choice represents a binary chain. 
	Now, the previous proof can also be phrased in terms of inversion tables.
	A valid traversal corresponds to an inversion table with $\ell$ chosen indices $i_0 < i_1 <\dots < i_{\ell-1}$ such that $t_{i_k} \geq i_{k-1}$ for $k\geq 1$.
\end{remark}

From Proposition~\ref{prop:delln} we observe the following equidistribution statistic between permutations and binary chains:
The distribution of the number of nodes on level $\ell$ in all decompressed binary chains $\DD(\CC_n)$ with $n$ internal nodes is equal to the distribution of increasing subsequences of length~$\ell$ in all permutations of~$n$. 
Let $s_{\ell,n}$ be the number of increasing subsequences of length $\ell$ in all permutations of $n$.
It is easy to derive its closed form as follows, by marking each increasing subsequence. 
First, choose $\ell$ out of $n$ positions.
Then choose $\ell$ out of $n$ numbers $\{1,2,\dots,n\}$ and fill the chosen positions with these in increasing order.
Finally, fill the remaining $n-\ell$ positions, with an arbitrary permutations of the remaining $n-\ell$ elements.
Therefore, we have
\begin{align}
	\label{eq:slnisdln}
	s_{\ell,n} = (n-\ell)! \binom{n}{\ell}^2 = \frac{n!}{\ell!}\binom{n}{\ell} = d_{\ell,n}.
\end{align}
We leave it as an open problem to find a bijection between these two quantities, which would then give an alternative and more combinatorial proof of Proposition~\ref{prop:delln}.
Note that there cannot be a bijection $\varphi : S_n \to \CC_n$ such that for all $\pi \in S_n$, the joint statistics of number of increasing subsequences is mapped to the joint distribution of nodes per level in $\DD(\varphi(\pi))$. 
For example, for $n=5$ there are $4$ permutations each having $5$ length $2$ increasing sequences, but no larger ones,
but there are $5$ elements in $\CC_5$ with $5$ nodes on level $2$ and none at larger levels.\footnote{We thank Martin Rubey for finding this example using SageMath.}

\section{Generalized addition chains and $k$-ary chains}
\label{sec:additionchains}

Let $m,n$ be positive integers.
An \emph{addition chain\footnote{In the literature of addition chains it is custom to use $n$ instead of $m$. However, in this paper, $n$ is always used for the size of the compressed tree, which corresponds to the length of the addition chain.} of length $n$ for $m$} is a sequence of numbers $(a_0, a_1, \dots, a_n)$ such that $a_0=1$, $a_n=m$, and for $i>0$ it holds that $a_{i} = a_{j_1} + a_{j_2}$, where $j_1,j_2 <i$; see~\cite[Section~4.6.3]{KnuthArt2}.
Addition chains play an important role in the efficient evaluation of powers~$x^m$, in the sense of minimizing the number of multiplications needed. 
For example, $x^{16}$ can be computed using~$4$ multiplications as follows: $x^2=x \cdot x$, $x^4 = x^2 \cdot x^2$, $x^8 = x^4 \cdot x^4$, and finally $x^{16}=x^8 \cdot x^8$.
The key property for fast multiplication is the minimal length $\ell(m)$ of an addition chain of $m$.
There are many open problems on addition chains.
The most famous is the Scholz--Brauer conjecture:
\begin{align}
	\label{eq:BrauerScholzConjecture}
	\ell(2^m-1) \leq m-1+\ell(m).
\end{align}
We refer to~\cite[Section~C6]{Guy2004Unsolved} and \cite[Section~4.6.3]{KnuthArt2} for more details and conjectures, as well as further references.

We are interested in a special class of addition chains, in which we always choose the immediate predecessor as a summand. 
A \emph{Brauer chain (or star chain) of length $n$ for $m$}, is an addition chain of length $n$ for $m$ such that $a_0=1$, $a_n=m$, and for $i>0$ it holds that $a_{i} = a_{i-1} + a_{j}$, where $j <i$; see~\cite{Guy2004Unsolved,Brauer1939Addition}.
Let $\ell^*(m)$ be the minimal length of a Brauer chain of~$m$. Then, Brauer~\cite{Brauer1939Addition} proved that~\eqref{eq:BrauerScholzConjecture} holds when $\ell$ is replaced by $\ell^*$.

\begin{proposition}
\label{prop:brauerbij}
Brauer chains of length $n$ are in bijection with binary chains with $n$ internal nodes.
In particular, the value $a_n$ is equal to the number of leaves in the decompressed binary tree, and therefore, the number of internal nodes is $a_n-1$.
\end{proposition}

\begin{proof} 
Take an arbitrary binary chain and write the number $1$ into the leaf. 
We now associate a number to each node recursively, starting with the one closest to the leaf. 
The number of a node is simply the sum of the two nodes its children are pointing at. 
This associates to the $i$th internal node (counted starting from the leaf) the value $a_i$, which is equal to the number of leaves in the decompressed tree.
Observe that $a_i$ is exactly the number of paths from a node to the leaf; compare Lemma~\ref{lem:Rtrajectories}.
This gives a well-known graphical representation of Brauer chains (and addition chains in general); see, e.g., \cite{Clift2011Addition} or \cite[pp.~480--481]{KnuthArt2}.
\end{proof}

Therefore, the Brauer chains for $m$ are exactly the binary trees with $m$ leaves that are compressible into chains of (compressed) size $n$.
These chain-compressible binary trees with $m$ leaves ($m-1$ internal nodes) are enumerated by~\OEISs{A008927} and 
the initial terms are given by
\begin{align}
	z + z^2 + z^3 + 2z^4 + 3z^5 + 6z^6 + 10z^7 + 20z^8 + 36z^9 + 70z^{10} + 130z^{11} + \dots. \label{eq:chaincompressiblebinary}
\end{align} 
Note that only the first $70$ terms seem to be known. 
Our methods do not seem to be capable of analyzing this sequence. 
We will instead focus on Brauer chains of fixed length $n$ and arbitrary $a_n$.

\smallskip

Our main result Theorem~\ref{theo:decompressed} in the case of binary trees admits then the following interpretation:
the expected value of $a_n$ for an addition chain of length $n$ chosen uniformly at random is equal to
\begin{align*}
		\frac{1}{2\sqrt{e \pi}} \frac{e^{2 \sqrt{n}}}{n^{1/4}} \left(1 + \LandauO\left(\frac{1}{\sqrt{n}}\right)\right).
	\end{align*}
	
Note that also our most general result for $k$-ary chains is connected with addition chains, in which addition is a $k$-ary operation.
Let $k\geq 2$ be a positive integer.
We define a \emph{$k$-Brauer chain (or $k$-star chain) of length $n$ of $m$} as a sequence of numbers $(a_0, a_1, \dots, a_n)$ such that $a_0=1$, $a_n=m$, and for $i>0$ it holds that $a_{i} = a_{i-1} + a_{j_1} + \dots + a_{j_{k-1}}$, where $j_{\ell} <i$ for $\ell= 1,\dots,k-1$.

For $k\geq3$ we need to be careful with the generalization of Proposition~\ref{prop:brauerbij}, since $k$-ary chains are in general in bijection with the specific implementation of the additions that gives a specific addition chain.
This implementation is often called \emph{addition scheme} or \emph{addition graph}~\cite{BrlekCasteranStrandh1991Additionschemes}. 
In applications, the specific implementation is indeed more important than the chain.
Note that one addition chain has in general multiple addition graphs, while each addition graph corresponds to exactly one addition chain.
Only for $k=2$ there is a one-to-one correspondence between addition graphs and addition chains.

In our context, even though addition is commutative, we distinguish the order in which the elements are added to reflect the order of the children in $k$-ary chains. 
Then, the expected value of $a_n$ in a $k$-Brauer graph of length~$n$ chosen uniformly at random is equal to the asymptotics of $\E(X_n)$ in Theorem~\ref{theo:decompressed}.
In order to have a bijection between $k$-Brauer chains and a DAG-class one would need to consider \emph{unordered} DAGs. However, these are more complicated to count and we are missing the analogous results of~\cite{GenitriniGittenbergerKauersWallner2016} on which we build on for unlabeled unordered graphs.

\section{Conclusion and further research directions}
\label{sec:conclusion}

We have analyzed the classical DAG compression algorithm focusing on its decompression properties.
In Theorem~\ref{theo:decompressed} we have shown that the expected decompressed size of a $k$-ary chain of size $n$ chosen uniformly at random is asymptotically equivalent to 
\[
	\kappa \, e^{c \sqrt{n}} n^{-1/4}
\]
for explicit constants $c$ and $\kappa$. 
To our knowledge, this is the first result on the decompression distribution of DAG compression. 

The key concept of this paper is that of a decompressed tree size given in Definition~\ref{def:decompressedtreesize}.
The underlying idea is that the compressed and decompressed world are strongly linked.
Thereby, statistics on trees can be used to define statistics on DAGs, and vice versa.
While the size is probably the most natural statistic, many other choices like the width or height, are possible and remain to be addressed.
Moreover, we generalized in~\eqref{eq:cgf} the important notion of exponential generating functions, to $d$-exponential generating functions, rescaling the $n$th coefficient by $(n!)^d$. 
This allowed us to derive a symbolic specification for decompressed $k$-ary chains using generating functions, which we could analyze with the tools from analytic combinatorics.

\begin{figure}[t]
	\centering
	\includegraphics[width=1\textwidth]{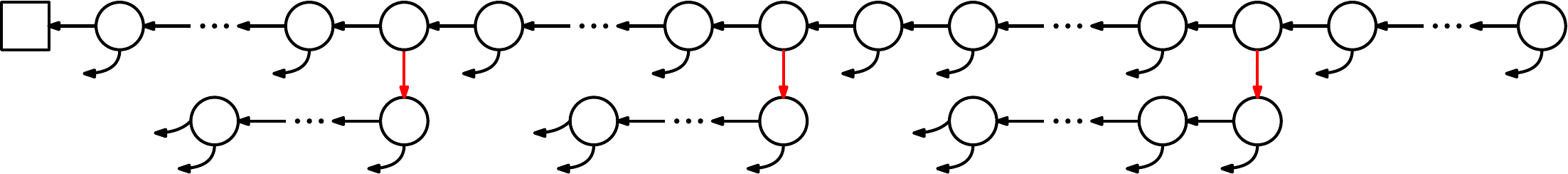}
	\caption{The spine with dangling pointers of a relaxed binary tree of right height at most $1$. The difference are the additional possible right-edges in the spine colored in red. The number of such DAGs with $n$ internal nodes is equal to the odd double factorials $(2n-1) \cdot (2n-3) \cdots 3 \cdot 1$.}
	\label{fig:rightheight1}
\end{figure}

\subsection{Open questions and further research}
\begin{itemize}
	\item 
	Can we further analyze the (limit) distribution of decompressed $k$-ary chains? 
	We have derived the asymptotic mean, however, we have no information on higher moments, like the variance.
	A first step is to check whether the derived expected value is indeed ``typical'' by, e.g., simulations.

	\item
	What is the expected decompressed size of other DAG classes?
	Interesting classes are chains with fixed out-degree sets, such as binary-ternary chains, as well as rooted DAGs with other constraints on the underlying spanning tree.
	A suitable first choice are relaxed and compacted trees of bounded right height~\cite{GenitriniGittenbergerKauersWallner2016}; see Figure~\ref{fig:rightheight1}.
It was shown that this class is enumerated by odd double factorials, which was later proved in~\cite{Wallner2017Bijection} bijectively. 

	\item
	What can we say about the decompressed size of unrestricted compacted trees?
	This is certainly the most interesting and challenging question, as we do not have a characterization in terms of generating functions.
	We will need different methods, e.g., building on the known bivariate recurrence from~\cite{ElveyPriceFangWallner2019Compacted}.

	\item
	On a more combinatorial level, 
	in~\eqref{eq:slnisdln} we have proved an equidistribution between the number of nodes on level $\ell$ in all decompressed binary chains $\DD(\CC_n)$ with $n$ internal nodes and the number of increasing subsequences of length~$\ell$ in all permutations of~$n$. It remains to find a combinatorial proof, e.g., in the form of a bijection on inversion tables.
	
	\item
	What is the number of binary trees with $n$ nodes that are compressible into chains? Only the first terms are known, yet no recurrence; see~\eqref{eq:chaincompressiblebinary} and~\OEISs{A008927}.

	\item
	Finally, now that the counting problem is solved, we can analyze parameters, such as the typical shape of large decompressed trees.
	For example, let $L_n$ be the random variable associated with the level of a random node in all decompressed binary chains $\DD(\CC_n)$ with $n$ internal nodes:
	\begin{align*}
		\PR(L_n = k) = \frac{[z^n q^k] D(z,q)}{[z^n]D(z,1)}.
	\end{align*}
	Then, by standard methods of analytic combinatorics one can show that $L_n$ satisfies a Gaussian limit law with mean and variance of order $\sqrt{n}$:
	\begin{align*}
		&&
		\E(L_n) &= \sqrt{n} + \LandauO(1) & 
		& \text{and} &
		\V(L_n) &= \frac{\sqrt{n}}{2} + \LandauO(1).
		&&
	\end{align*}
	This follows directly from \cite[Example~IX.33]{flaj09} and the therein discussed perturbation of saddle-point asymptotics, and, with more work, one has access to a full asymptotic expansion.
	The same method is applicable to $k$-ary chains and other parameters.

\end{itemize}

\section*{Acknowledgement}

This research was partially supported by the Austrian Science Fund (FWF) P~34142 and AST1535024.

\section*{Previous version}

This work extends the extended abstract~\cite{Wallner2025Decompressed} that appeared in and was presented at the European Conference on Combinatorics, Graph Theory and Applications 2025 (Eurocomb'25) in Budapest, Hungary. 

\addcontentsline{toc}{section}{References}
\bibliographystyle{mybiburl}
\bibliography{Bibliography}
\label{sec:biblio}

\end{document}